\documentclass{iopart}
\usepackage{iopams}
\pdfoutput=1
\usepackage{epstopdf}
\usepackage{latexsym,color,wrapfig}
\usepackage{diagbox}
\usepackage{comment}
\usepackage{amsfonts}
\usepackage{graphicx}
\usepackage{subfig}
\usepackage{url}
\def\sqr#1#2{$\vcenter{\hrule height.#2pt
\hbox{\vrule width.#2pt height#1pt \kern#1pt \vrule width.#2pt} \hrule
height.#2pt}$}
\def\qed{\sqr53}
\def\O{\Omega}
\def\M{\Lambda}

\def\U{\mathcal{U}}
\def\pp{{\prime\prime}}
\def\p{\prime}

\newtheorem{theorem}{Theorem}
\newtheorem{proposition}{Proposition}

\newtheorem{corollary}{Corollary}
\newtheorem{remark}{Remark}

\def\b{\textbf{b}}

\begin{document}
\title[The number of  rational links with a given deficiency]{The number of oriented rational links with a given deficiency number}
\author{Yuanan Diao$^{\ast}$$^{\ddag}$, Michael Lee Finney$^{\ddag}$, Dawn Ray$^{\ddag}$}
\address{
$^{\ddag}$ Department of Mathematics and Statistics\\
University of North Carolina at Charlotte\\
 Charlotte, NC 28223, USA\\
$^{\ast}$ Correspondance E-mail: ydiao@uncc.edu}

\begin{abstract}
Let $\U_n$ be the set of  un-oriented and rational links with crossing number $n$, a precise formula for $|\U_n|$ was obtained by Ernst and  Sumners in 1987. In this paper, we study the enumeration problem of oriented rational links. Let $\M_n$ be the set of oriented rational links with crossing number $n$ and let $\M_n(d)$ be the set of oriented rational links with crossing number $n$ ($n\ge 2$) and deficiency $d$. In this paper,  we derive precise formulas for $|\M_n|$ and $|\M_n(d)|$ for any given $n$ and $d$ and show that 
$$
\M_n(d)=F_{n-d-1}^{(d)}+\frac{1+(-1)^{nd}}{2}F^{(\lfloor \frac{d}{2}\rfloor)}_{\lfloor \frac{n}{2}\rfloor -\lfloor \frac{d+1}{2}\rfloor},
$$
where $F_n^{(d)}$ is the convolved Fibonacci sequence. 
\end{abstract}

\section{Introduction}

The enumeration of knots and links is a classical problem in knot theory. Examples include knot enumeration tables in any typical knot theory textbook such as \cite{A,BZ,Cromwell}, and the exact counting of the number of knots/links with a given crossing number such as \cite{Co,Doll,Hoste1998}. Of particular interest to us in this paper is the enumeration of rational links. Let $\U_n$ and $\M_n$ be the sets of un-oriented and oriented rational links with crossing number $n$ respectively.
From the  work of  Ernst and Sumners in \cite{Ernst1987},  one can obtain $|\U_n|=2^{n-3}+2^{\lfloor\frac{n-3}{2}\rfloor}$ for $n\ge 2$, where $\lfloor x\rfloor$ denotes the greatest integer that is less than or equal to $x$. It is known that rational links are invertible, that is, changing the orientations of all components in a rational link will not change the link type. However, in the case that an oriented rational link $\mathcal{L}$ has two components, changing the orientation of only one component of $\mathcal{L}$ may result in a rational link that is topologically different from $\mathcal{L}$ (as oriented links). $\mathcal{L}$ is said to be {\em strongly invertible} if changing the orientation of one of its component does not change its link type. This difference is not observed when $\mathcal{L}$ is un-oriented. Thus in the counting of un-oriented rational links it is as if we are treating all oriented rational links with two components as strongly invertible ones. With this observation, one would be able to find a precise formula for $|\M_n|$, provided that one can determine the precise number of strongly invertible rational links with crossing number $n$. Indeed this is one of the results obtained in this paper. However this is a relatively easy result. 

\medskip
The main objective of this paper concerns the enumeration of  oriented rational links with a given crossing number as well as a given {\em deficiency}  number. For an oriented link $\mathcal{L}$, its deficiency number (or just deficiency for short) $d(\mathcal{L})$ is defined as 
\begin{equation}\label{deficiency_def}
d(\mathcal{L})=Cr(\mathcal{L})-2g(\mathcal{L})-\b(\mathcal{L})-\mu(\mathcal{L})+2,
\end{equation}
where $Cr(\mathcal{L})$ is the crossing number of $\mathcal{L}$, $g(\mathcal{L})$ is the genus of $\mathcal{L}$, $\b(\mathcal{L})$ is the braid index of $\mathcal{L}$ and $\mu(\mathcal{L})$ is the number of components of $\mathcal{L}$. The concept of the deficiency number is related to the additivity of crossing numbers under the connected sum operation. An old and open conjecture in knot theory states that the crossing number of the connected sum of two links is the sum of the crossing numbers of the two links. It has been shown in \cite{Diao2004} that 
$$
Cr(\mathcal{L}_1)+Cr(\mathcal{L}_2)-d(\mathcal{L}_1)-d(\mathcal{L}_2)\le Cr(\mathcal{L}_1\#\mathcal{L}_2)\le Cr(\mathcal{L}_1)+Cr(\mathcal{L}_2),
$$ 
hence this conjecture is true if both links are of deficiency zero.
Let $\M_n(d)$ be the set of oriented rational links with crossing number $n$ and deficiency $d$ ($n\ge 2$ and it is necessary that $d\le (n-2)/2$). It is apparent that
$\{\M_n(0), \M_n(1), ..., \M_n(\lfloor \frac{n-2}{2}\rfloor)\}$ is a partition of  $\M_n$ so we have
\begin{equation}\label{partition_sum}
\sum_{d=0}^{\lfloor \frac{n-2}{2}\rfloor}|\M_n(d)|=|\M_n|.
\end{equation}

\medskip
In \cite{DiaoErnst2005}, it was estimated that $|\M_n(0)|$ is bounded below by $e^{\frac{\ln 3}{7} n}$. In this paper, we will derive a precise formula of $|\M_n(d)|$ for any given $n$ and $d$. Furthermore, we show that 
\begin{equation}\label{main}
\M_n(d)=F_{n-d-1}^{(d)}+\frac{1+(-1)^{nd}}{2}F^{(\lfloor \frac{d}{2}\rfloor)}_{\lfloor \frac{n}{2}\rfloor-\lfloor \frac{d+1}{2}\rfloor},
\end{equation}
where $F_n^{(d)}$ is the convolved Fibonacci sequence defined by $F_{n}^{(0)}=F_n$ and 
$F_{n}^{(d+1)}=\sum_{i=0}^n F_{i}F_{n-i}^{(d)}$ with $\{F_0,F_1,F_2,F_3,F_4,F_5,...\}=\{0,1,1,2,3,5,...\}$ being the Fibonacci sequence. In the particular case of $d=0$, we have
$|\M_n(0)|=F_{n-1}+F_{\lfloor \frac{n}{2}\rfloor}$. Using this result, the lower bound on $|\M_n(0)|$ is improved from $e^{\frac{\ln 3}{7} n}=(\sqrt[7]{3})^n\approx 1.17^n$ to approximately $(1.618^{n-1}+1.272^n)/\sqrt{5}$.
 
\medskip
We organize the rest of the paper in the following way. In the next section, we introduce the rational links in a  special form of 4-plats (called {\em preferred standard form} or PS form for short) which we will use throughout this paper for our purpose of enumeration. In Section  \ref{Seifert_Decomposition}, we discuss the Seifert circle decompositions of 4-plats in the PS form. In  Sections \ref{enu_ori} and \ref{enu_def} we derive the combinatorial formulas for $|\M_n|$ and $|\M_n(d)|$ respectively in that order. In the last section, we give the proof for (\ref{main}).
 
\medskip
\section{Rational Links as 4-plats in the Preferred Standard Form}

We shall assume that our readers have some basic knowledge of knot theory and its usual terminology, as well as some basic knowledge of the rational links. \cite{BZ} is a good reference for readers who are not familiar with these subjects. Throughout the rest of this paper, all rational links will be oriented unless otherwise stated. For the purpose of this paper, we would like to adhere to a particular presentation of the rational links. Let $p$ and $q$ be two positive integers with $\rm{gcd}(p,q)=1$ and $0<p<q$. Let $(a_1,a_2,...,a_{2k+1})$ be the (unique) vector of odd length with positive integer entries such that 
$$
\displaystyle\frac{p}{q}=\displaystyle\frac{1}{a_1+\frac{1}{a_2+\frac{1}{....\frac{1}{a_{2k}+\frac{1}{a_{2k+1}}}}}}.
$$
For the sake of convenience we will write the above as $p/q=[a_1,a_2,...,a_{2k+1}]$. In this paper we shall represent a rational link corresponding to $p/q$ as a 4-plat as shown in Figure \ref{Figure1} for the case of $56/191=[3,2,2,3,3]$, with the long arc in the 4-plat at the bottom and oriented from right to left (we can do so because oriented rational links are known to be invertible). Furthermore, we want the first string of horizontal crossings from the left of the 4-plat to be corresponding to $a_1$ and the strand connected to the long arc at the bottom being the under strand. For the sake of convenience we shall call a 4-plat satisfying these conditions a 4-plat in the {\em preferred standard form} (PS for short) in this paper. 

\medskip
A 4-plat $\mathcal{L}(p/q)$ has two components if and only if $q$ is even. When $\mathcal{L}(p/q)$ has two components, the orientation of the component containing the long arc has already been determined by our choice of the orientation of the long arc (from right to left), but the orientation of the other component has two choices, one of which will make the left most crossing in the 4-plat positive and the other one will make it negative. We shall denote these two 4-plats by $\mathcal{L}_+(p/q)$ and $\mathcal{L}_-(p/q)$ respectively. For example, the 4-plat shown in Figure \ref{Figure1} is $\mathcal{L}_-(56/191)$. Thus for a 4-plat in the PS form corresponding to $p/q=[a_1,a_2,...,a_{2k+1}]$, the orientation of the over strand at the left most crossing (hence the 4-plat as an oriented link) is either already determined (in the case that $q$ is odd), or it can be either $\mathcal{L}_+(p/q)$ or $\mathcal{L}_-(p/q)$ (but nothing else, since the orientation of the components completely determine the 4-plat).

\medskip
\begin{figure}[!hbt]
\begin{center}
\includegraphics[scale=.35]{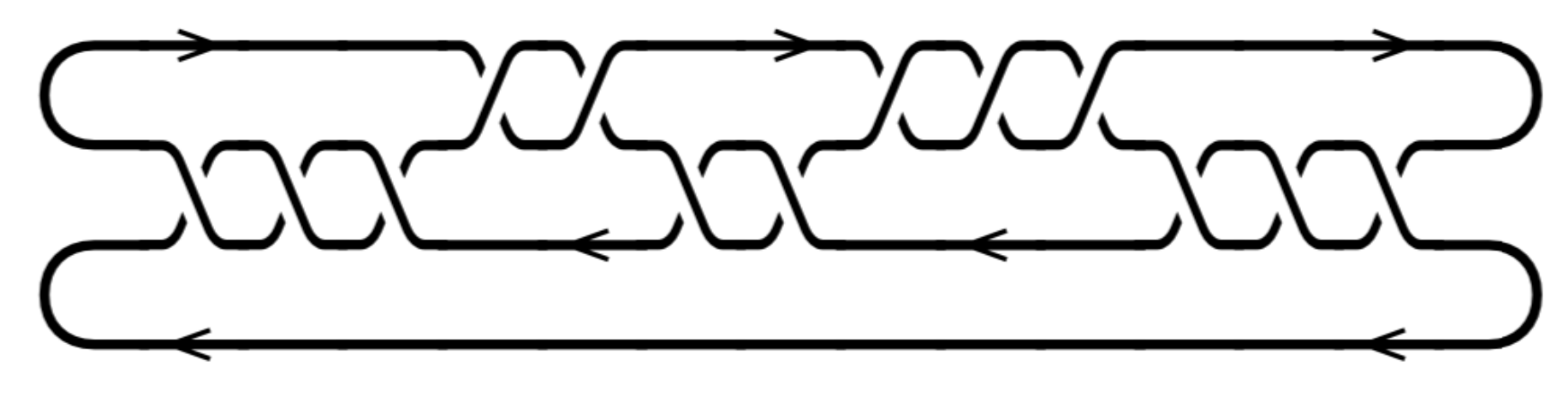}
\end{center}
\caption{The rational link $\mathcal{L}_-(56/191)$ as a 4-plat in the preferred standard form.
\label{Figure1}}
\end{figure}

Given a 4-plat $\mathcal{L}$ in the PS form, if we rotate it by 180 degrees around the $y$ axis and changing orientations of the components, then the resulting 4-plat $\mathcal{L}^\p$ (which we will call the {\em reversal} of $\mathcal{L}$) is equivalent to the inverse of $\mathcal{L}$, which in turn is equivalent to $\mathcal{L}$ since rational links are known to be invertible. Thus $\mathcal{L}^\p\sim \mathcal{L}$.

\medskip
For the sake of convenience, we state the following two classical results concerning rational links due to Schubert \cite{Schubert}. 

\medskip
\begin{proposition}{\em \cite{Schubert}}\label{Proposition1} Suppose that rational tangles with fractions $p/q$ and $p^\p/q^\p$ are given with $\rm{gcd}(p,q)=\rm{gcd}(p^\p,q^\p)=1$. If $\mathcal{L}(p/q)$ and 
$\mathcal{L}(p^\p/q^\p)$ are the corresponding rational links obtained by taking denominator closures of these tangles, then $\mathcal{L}(p/q)$ and 
$\mathcal{L}(p^\p/q^\p)$ are topologically equivalent (as un-oriented links) if and
only if (i) $q=q^\p$ and (ii) either $p\equiv p^\p$ (mod $q$) or $p\cdot p^\p\equiv  1$ (mod $q$).
\end{proposition}

\begin{proposition}{\em \cite{Schubert}}\label{Proposition2} Suppose that two oriented rational links $\tilde{\mathcal{L}}(p/q)$ and 
$\tilde{\mathcal{L}}(p^\p/q^\p)$ are obtained by taking denominator closures of two orientation compatible tangles corresponding to rational numbers $p/q$ and $p^\p/q^\p$ with $p$, $p^\p$ odd, $\rm{gcd}(p,q)=\rm{gcd}(p^\p,q^\p)=1$, then $\tilde{\mathcal{L}}(p/q)$ and 
$\tilde{\mathcal{L}}(p^\p/q^\p)$ are topologically equivalent if and only if (i) $q=q^\p$ and (ii) either $p\equiv p^\p$ (mod $2q$) or $p\cdot p^\p\equiv  1$ (mod $2q$).
\end{proposition}

\medskip
For our purpose we need the result stated in the following theorem, which is a corollary of Propositions \ref{Proposition1} and \ref{Proposition2}. For the sake of completeness we include a proof.

\begin{theorem}{\em \cite{Schubert}}\label{Theorem1}
Every oriented rational link $\mathcal{L}$ can be represented by at most two 4-plats in the PS form. In the case that  $\mathcal{L}$ can be represented by two distinct 4-plats in the PS form, these two 4-plats are the reversal of each other.
\end{theorem}

\begin{proof}
Assume that $\mathcal{L}$ is represented by a 4-plat in the PS form corresponding to some rational number $p/q$ with $0<p<q$, $\rm{gcd}(p,q)=1$. The reversal of this 4-plat is also a preferred standard 4-plat in the form of $\mathcal{L}(p^\p/q)$ for some rational number $p^\p$ with $0<p^\p<q$, $\rm{gcd}(p^\p,q)=1$. We need to show that there are no other 4-plats in the PS form that are equivalent to $\mathcal{L}$. Assume the contrary, then $\mathcal{L}$ can be represented by another PS form 4-plat corresponding to some rational number $p^\pp$ with $0<p^\pp<q$, $\rm{gcd}(p^\pp,q)=1$. There are only two cases for us to consider.

\smallskip
Case 1. $p\not= p^\p$. We have $pp^\p\equiv 1$ mod$(q)$ and $pp^\pp\equiv 1$ mod$(q)$ by Proposition \ref{Proposition1}. It follows that $pp^\p\equiv pp^\pp$ mod$(q)$. Since $\rm{gcd}(p,q)=1$ and $0<p^\p,\ p^\pp<q$, it follows that $p^\p=p^\pp$. Similarly, $p^\p p\equiv p^\p p^\pp$ mod$(q)$ and $p=p^\pp$. Thus we have $p=p^\p$, which is a contradiction.

\smallskip
Case 2. $p=p^\p$. We have $p^2\equiv 1$ mod$(q)$ and $pp^\pp\equiv 1$ mod$(q)$, hence $p=p^\pp$ as well. Notice that $p=p^\p$ happens only if $p/q=[a_1,a_2,...,a_{k},\alpha,a_{k},...,a_2,a_1]$ for some positive integers $a_1,a_2,...,a_{k}$ and $\alpha$, that is, the odd length continued fraction decomposition of $p/q$ is a palindrome.  If $\mathcal{L}$ has only one component, there is only one way to draw the 4-plat hence $\mathcal{L}(p^\pp/q)$ must be the same as $\mathcal{L}(p/q)$, which is a contradiction. Thus the only case left to prove is when $\mathcal{L}$ has two components. If $\mathcal{L}(p/q)$ is $\mathcal{L}_+(p/q)$ ($\mathcal{L}_-(p/q)$) while its reversal is $\mathcal{L}_-(p/q)$ ($\mathcal{L}_+(p/q)$), then $\mathcal{L}(p^\pp/q)$ must be either $\mathcal{L}_+(p/q)$ or $\mathcal{L}_-(p/q)$ since $p^\pp=p$, which is also a contradiction. The final case is when $\mathcal{L}(p/q)$ is $\mathcal{L}_+(p/q)$ ($\mathcal{L}_-(p/q)$) while its reversal is also $\mathcal{L}_+(p/q)$ ($\mathcal{L}_-(p/q)$). Say $\mathcal{L}(p/q)=\mathcal{L}_+(p/q)$ (the case that $\mathcal{L}(p/q)=\mathcal{L}_-(p/q)$ can be similarly discussed) and its reversal is also $\mathcal{L}_+(p/q)$. This means that the right most crossing in $\mathcal{L}_+(p/q)$ is also positive. $\mathcal{L}(p^\pp/q)$ must be $\mathcal{L}_-(p/q)$ in order to be different, which means that it is obtained from $\mathcal{L}_+(p/q)$ by changing the orientation of the component containing the over strand at the left most crossing (same can be said about the right most crossing). However, by Remark \ref{block_remark} on the {\em blocks} of Seifert circle decomposition of oriented link diagrams in Section \ref{Seifert_Decomposition}, $\mathcal{L}_+(p/q)$ and $\mathcal{L}_-(p/q)$ cannot be equivalent in this case, and we arrive at a contradiction as desired. \qed
\end{proof}
 
 \medskip
An implication of Theorem \ref{Theorem1} (and its proof) is that if a rational link $\mathcal{L}(p/q)$ is strongly invertible, then the odd length positive continued fraction decomposition of $p/q$ is palindromic, $\mathcal{L}_+(p/q)$ and $\mathcal{L}_-(p/q)$ must be the reversal of each other, and the right most crossing of $\mathcal{L}_+(p/q)$ ($\mathcal{L}_-(p/q)$) must be negative (positive). 

\medskip
\section{Seifert Circle Decompositions of 4-plats in the PS Form}\label{Seifert_Decomposition}

Consider a rational link $\mathcal{L}(p/q)$ with $p/q=[a_1,a_2,...,a_{2m+1}]$ ($a_i>0$ for each $i$) as a 4-plat in the PS form. Since all crossings corresponding to a given $a_i$ have the same crossing sign, we can define a signed vector $[b_1,b_2,...,b_{2m+1}]$ where $b_i =\pm a_i$ with its sign given by the crossing sign of the crossings corresponding to $a_i$. For example, for $\mathcal{L}(5075/17426)$ with the orientation shown in Figure \ref{2bridgeone} we obtain the signed vector $[3,2,3,3,-1,-2,-3,4,-4]$ (hence this link is $\mathcal{L}_+(5075/17426)$ by our definition). Let us group the consecutive $b_j$'s with the same signs together and call these groups {\em blocks}, which we will denote by $B_1$, $B_2$, ...,  and so on. For example, the link given in Figure \ref{2bridgeone} has four blocks $B_1=(3,2,3,3)$, $B_2=(-1,-2,-3)$,  $B_3=(4)$ and $B_4=(-4)$. 

\begin{figure}[htb!]
\begin{center}
\includegraphics[scale=0.3]{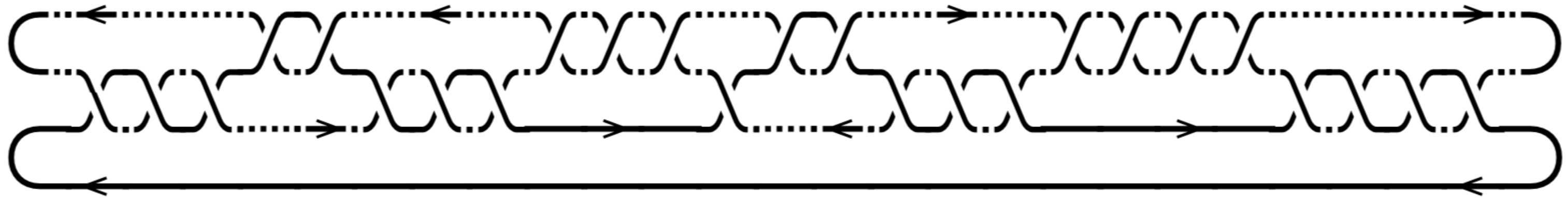}\\
\vspace{0.2in}
\includegraphics[scale=0.3]{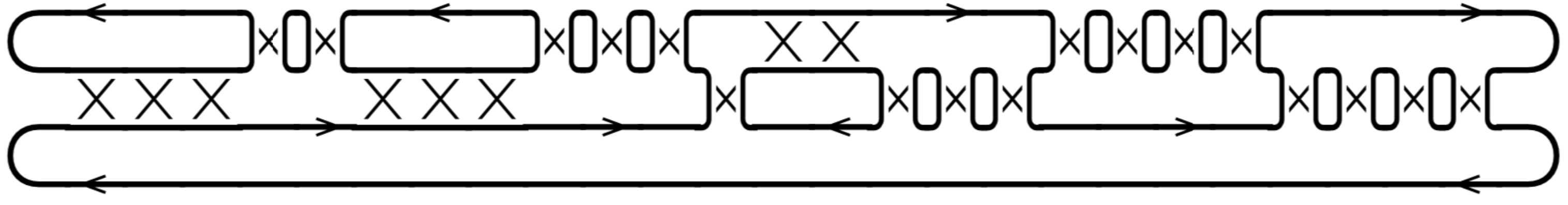}
\end{center}
\caption{Top: The 4-plat  given by $5075/17426=[3,2,3,3,1,2,3,4,4]$ with the shown orientation; Bottom: Its corresponding Seifert circle decomposition.}
\label{2bridgeone}
\end{figure}

\medskip
In the Seifert circle decomposition of $\mathcal{L}(p/q)$, we call the Seifert circle $C$ that contains the bottom strand of the 4-plat the {\em large} Seifert circle, a Seifert circle that shares at least a crossing with $C$ a {\em medium} Seifert circle and a Seifert circle that does does not share a crossing with $C$ a {\em small} Seifert circle. Each medium or small Seifert circle is either contained inside or outside of $C$. Using the well known facts that $\mathcal{L}(p/q)$ is alternating and alternating link diagrams are homogeneous, it is easy to see that the crossings inside $C$ all have negative signs and crossings outside of $C$ all have positive signs. Thus, a 4-plat $\mathcal{L}(p/q)$ in the PS form can be reconstructed from its Seifert circle decomposition, as long as the smoothed crossings are marked in the Seifert circle decomposition.

\medskip
In the following, we detail how a block in the signed vector of a 4-plat $\mathcal{L}(p/q)$ in the PS form is related to its Seifert circle decomposition.

\medskip
(i) There is only one block and it is positive; Here there are two possibilities, either the block is of length one or $m \geq 1$.  If the block is of length one, then $\mathcal{L}_+(1/q)=[b_1]$ is a knot if $b_1$ is odd and a link if $b_1$ is even and there is one medium Seifert circle. Otherwise the link is of the form $\mathcal{L}_+(p/q)=[b_1, b_2,..., b_{2m+1}]$ where each $b_{2m}=2\rho_m$, $\rho_m \in \mathbb{Z}$ represents $2\rho_m-1$ small Seifert circles and each $b_{2m+1}$ represents one medium Seifert circle. See Figure \ref{one_positive_block} for an example, where it shows the Seifert circle decomposition corresponding to the 4-plat with one positive block $[2,4,3,2,1,2,4]$. Notice that it has an odd length and the entries at even positions are even, which contribute a total of $(4-1)+(2-1)+(2-1)=5$ small Seifert circles. Since there are 4 entries at odd positions, these entries contribute a total of 4 medium Seifert circles. Therefore the total number of small/medium Seifert circles is $4+5=9$ and the link has a total of 10 Seifert circles including the large Seifert circle.

\begin{figure}[htb!]
\begin{center}
\includegraphics[scale=0.4]{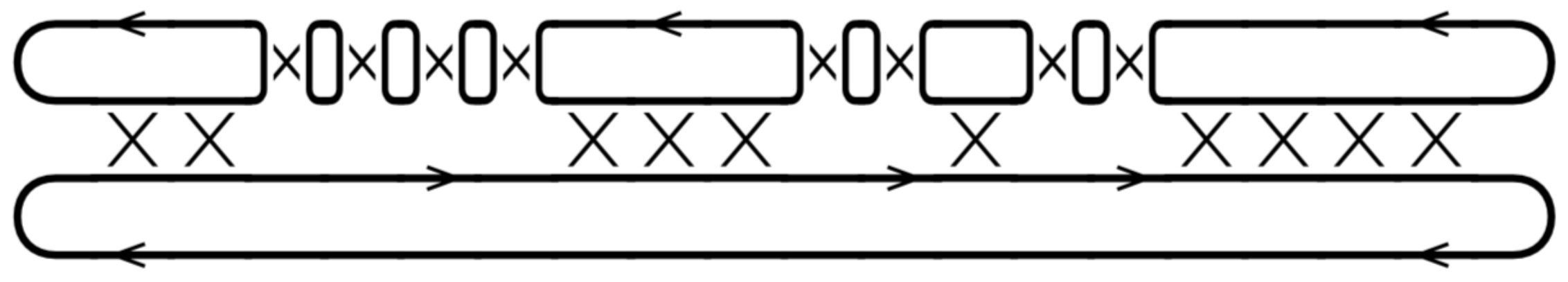}
\end{center}
\caption{The Seifert circle decomposition corresponding to a rational link with one positive block $[2,4,3,2,1,2,4]$.}
\label{one_positive_block}
\end{figure}

\medskip
(ii) There is only one block and it is negative; Again there are two possibilities.  If the block has one entry $\mathcal{L}_2(1/q)=[b_1]$ then $|b_1|=2\rho_1$ is even and it contributes $2\rho_1-1$ small Seifert circles.  Otherwise, the link is of the form $\mathcal{L}_2(p/q)=[b_1, b_2,..., b_{2m+1}]$ where $|b_{2j+1}|=2\rho_{2j+1}+1$ is odd and contributes $2\rho_{2j+1}$ small Seifert circles if $j=0$ or $j=m$, and $|b_{2j+1}|=2\rho_{2j+1}$ is even and contributes $2\rho_{2j+1}-1$ small Seifert circles if  $1\le j\le m-1$, while each $b_{2j}$, $1\le j\le m$, contributes one medium Seifert circle. See Figure \ref{one_negative_block} for an example, where it shows the Seifert circle decomposition corresponding to the 4-plat with one negative block $[-3,-2,-2,-3,-4,-1,-2,-3,-1]$. Notice that it also has an odd length and the entries at odd indexed positions are even except the first and the last, which contribute a total of $(3-1)+(2-1)+(4-1)+(2-1)+(1-1)=7$ small Seifert circles. Since there are 4 entries at even indexed positions, these entries contribute a total of 4 medium Seifert circles. Therefore the total number of small/medium Seifert circles is $4+7=11$ and the link has a total of 12 Seifert circles including the large Seifert circle.

\begin{figure}[htb!]
\begin{center}
\includegraphics[scale=0.4]{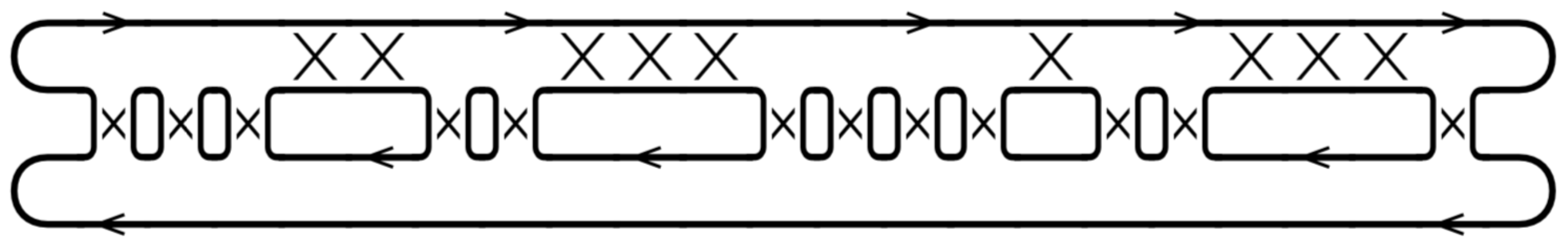}
\end{center}
\caption{The Seifert circle decomposition corresponding to a 4-plat with one negative block $[-3,-2,-2,-3,-4,-1,-2,-3,-1]$.}
\label{one_negative_block}
\end{figure}

\medskip
(iii) There are at least two blocks; In this case the positive and negative blocks will appear in an alternating manner. Any negative block regardless of its placement in the vector will follow the same rules as outlined in (ii).  For a positive block there are several cases depending on the location of the block. (a) If a positive block is the first one or the last one in the vector, then it must have an even length. A positive first block will be of the form $\langle b_1, b_2, ....b_{2j} \rangle$ for some $j\ge 1$, each $b_{2i}$ with $i<j$ (when $j>1$) is even and contributes $b_{2i}-1$ small Seifert circles, $b_{2j}$ is odd and contributes $b_{2j}-1$ small Seifert circles, while each $b_{2i+1}$ ($0\le i\le j-1$) contributes one medium Seifert circle. If the last block is positive, then it is a positive first block if we take the reversal so the discussion about the positive block can be applied. See Figure \ref{multiple_blocks} for an example. 
(b)  If a positive block  is in the middle of the vector, then the discussion is similar to the case of a single negative block. See Figure \ref{multiple_blocks} for an example.

\medskip
\begin{figure}[htb!]
\begin{center}
\includegraphics[scale=0.3]{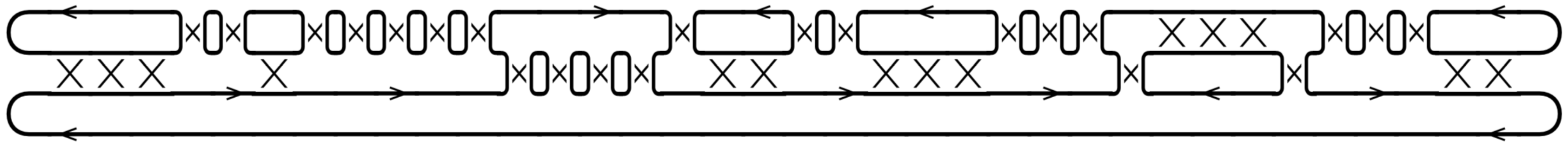}
\end{center}
\caption{The Seifert circle decomposition corresponding to the rational link with multiple blocks $[3,2,1,5,-4,1,2,2,3,3,-1,-3,-1,3,2]$.}
\label{multiple_blocks}
\end{figure}

\medskip
Let us call a Seifert circle decomposition obtained from a 4-plat in the PS form an {\em R-decomposition}. An R-decomposition can be constructed from a vector consisting of signed blocks satisfying conditions (i) to (iii) above. Based on the above discussions, we see that the   4-plats in the PS form,  as well as their corresponding R-decompositions of, can be divided into the following four types.

\medskip
\textit{Type I}: The left most crossing is positive and the right most crossing is negative. The corresponding Seifert circle decomposition has an even number of blocks, with the first block being positive and the last block being negative. 

\medskip
\textit{Type II}: The left most crossing is negative and the right most crossing is positive. The corresponding Seifert circle decomposition has an even number of blocks, with the first block being negative and the last block being positive.

\medskip
\textit{Type III}: The left most crossing and the right most crossing are both positive. The corresponding Seifert circle decomposition has an odd number of blocks, with the first block and last block both being positive.

\medskip
\textit{Type IV}: The left most crossing and the right most crossing are both negative. The corresponding Seifert circle decomposition has an odd number of blocks, with the first block and last block both being negative.

\medskip
We shall use $R_n^{I}$, $R_n^{II}$, $R_n^{III}$ and $R_n^{IV}$ to denote the sets of Type I, II, III, IV R-decompositions respectively and use $RS_n^{I}$, $RS_n^{II}$, $RS_n^{III}$ and $RS_n^{IV}$ the corresponding subsets within each that are symmetric (with respect to the reversal operation).

\medskip
\begin{remark}{\em 
Notice that the reversal of a Type I (Type II) 4-plat is of Type II (Type I) hence $RS_n^{I}=RS_n^{II}=\emptyset$ (see Figure \ref{Figure9} for an example), while the reversal of a Type III (Type IV) 4-plat remains a Type III (Type IV) 4-plat. Also, the R-decomposition of a 4-plat $\mathcal{L}_+$ must be of Type I or III, while the R-decomposition of a 4-plat $\mathcal{L}_-$ must be of Type II or IV.}
\end{remark}

\begin{figure}[!ht]
\begin{center}
\includegraphics[scale=.3]{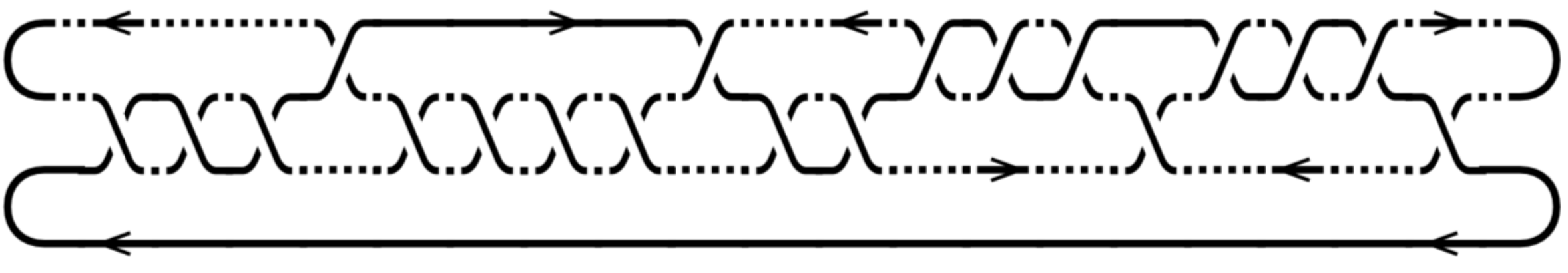}\\
\vspace{0.1in}
\includegraphics[scale=.3]{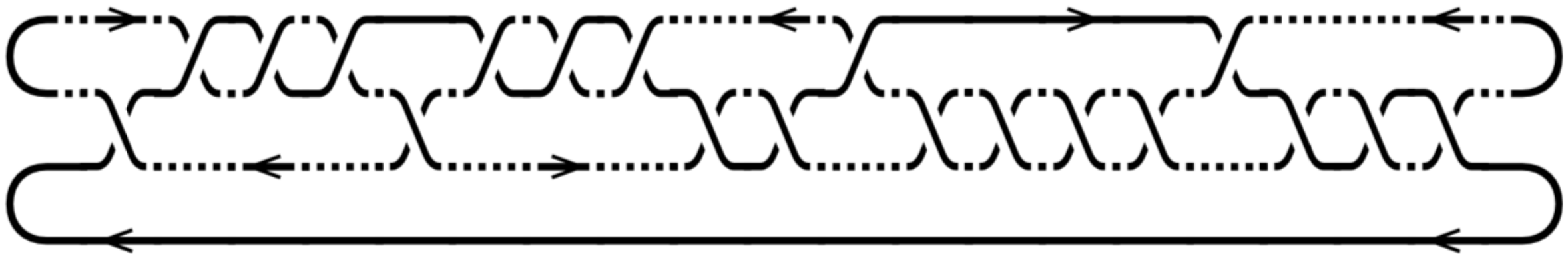}
\end{center}
\caption{The 4-plat with signed vector $[3,1,-4,1,2,3,-1,-3,-1]$ (Top) is of Type I while its reversal (Bottom) is of Type II.
\label{Figure9}}
\end{figure}

\medskip
\begin{remark}\label{R_counting}{\em 
Since there is a unique correspondence between a 4-plat in the PS form and its R-decomposition, an R-decomposition and its reversal correspond to one and only one (oriented) rational link by Theorem \ref{Theorem1}. Thus we can obtain the precise count of oriented rational links with a given crossing number from the number of R-decompositions with the same crossing number. More specifically, let $\mathcal{L}$ be a 4-plat in the PS form and $\mathcal{L}^\p$ be its reversal, and let $R$ and $R^\p$ be their R-decompositions respectively, then if $R$ and $R^\p$ are not identical ({\em i.e.}, they are not symmetric with respect to the $y$-axis), only one of them can be counted in  order to avoid over counting. This leads to 
\begin{equation}\label{counting1}
|\Lambda_n|=\frac{|\O_n|+|\O_n^\p|}{2},
\end{equation}
where $\O_n$ is the set of all R-decompositions with $n$ crossings and $\O_n^\p$ is the set of all symmetric R-decompositions with $n$ crossings.
}
\end{remark}

\medskip
\begin{remark}\label{R_counting_2}{\em 
Notice that the reversal operation is a one-to-one correspondence between Type I and Type II R-decompositions, so we have$|R^I_n|=|R^{II}_n|$. On the other hand, the mapping that takes a rational link to its  mirror image induces a one to one mapping between the signed vectors of (symmetric) Type III and (symmetric) Type IV rational links. For example, the mirror image of the Type III rational link with signed vector $[3,2,1,5,-4,1,2,2,3,3,-1,-3,-1,3,2]$ is $[-1,-2,-2,-1,-5,4,-1,-2,-2,-3,-3,1,3,1,-3,-1,-1]$. This then defines a one to one relation between $R^{III}$ ($RS^{III}$) and $R^{IV}$ ($RS^{IV}$). It follows that $|R^{III}_n|=|R^{IV}_n|$ and $|RS^{III}_n|=|RS^{IV}_n|$. Thus 
\begin{equation}\label{O_1}
|\O_n|= |R^I_n|+|R^{II}_n|+|R^{III}_n|+|R^{IV}_n|=2|R^I_n|+2|R^{III}_n|
\end{equation}
and
\begin{equation}\label{O_2}
|\O^\p_n|=|RS^{III}_n|+|RS^{IV}_n|=2|RS^{III}_n|.
\end{equation}
}
\end{remark}

\medskip
Substituting (\ref{O_1}) and (\ref{O_2}) into (\ref{counting1}), we obtain the following equation, which will serve as our main counting tool in Section 5.
\begin{equation} 
|\Lambda_n|=|R^I_n|+|R^{III}_n|+|RS^{III}_n|. \label{counting2}
\end{equation}

\section{The Enumeration of Oriented Rational Links}\label{enu_ori}

Let $\Lambda^1_n$ ($\Lambda^2_n$) be the set of oriented rational links with crossing number $n$ and one (two) component(s). We have $|\M_n|=|\Lambda^1_n|+|\Lambda^2_n|$.
Let $TK_n$ be the number of un-oriented rational knots (links with one component) with crossing number $n$, and $TL_n$ be the number of two component un-oriented rational links with crossing number $n$. Notice that these numbers are denoted by $TK^\ast_n$ and $TL^\ast_n$ respectively in \cite{Ernst1987}. Furthermore, let $TL^s_n$ be the number of two component un-oriented rational links with crossing number $n$ which can be represented by symmetric 4-plats in the PS form (but with the orientation information ignored). 
In the case of a rational knot, its orientation plays no important role since it is invertible. Thus $|\Lambda^1_n|=TK_n$. In the case of a two component rational link (represented by a  4-plat in the  PS form), the two 4-plats obtained by choosing different orientation of the component not containing the long arc are equivalent as oriented links if and only if they are strongly invertible. Thus $|\Lambda^2_n|=2TL_n-|\Lambda^\ast_n|$, and $|\M_n|=|\Lambda^1_n|+|\Lambda^2_n|=TK_n+2TL_n-|\M^\ast_n|$, where $\Lambda^\ast_n$ is the set of strongly invertible 4-plats with $n$ crossings. 

\medskip
\begin{remark}\label{block_remark}{\em 
By a recent result established in \cite{Diao_Pham2020}, the number of positive blocks and the number of negative blocks in the signed vector of an oriented 4-plat in the PS form are invariants among the minimum diagrams of the same link. Since the number of positive (negative) blocks in a Type III (Type IV) 4-plat is one more than the number of negative (positive) blocks in it, and $\mathcal{L}_+(p/q)$ and $\mathcal{L}_-(p/q)$ are of different types when $q$ is even and $\mathcal{L}(p/q)$ has two components, it follows that if $\mathcal{L}(p/q)$ is strongly invertible, then it must of Type I or Type II. By Theorem 8.1 in \cite{Ka-La}, if $\mathcal{L}(p/q)$ is strongly invertible, then $p/q=[a_1,a_2,...,a_k,\alpha,a_k,...,a_2,a_1]$ for some integers $a_1>0$, ..., $a_k>0$, $\alpha> 0$. This result has been strengthened in \cite{Diao_Pham2020} to the following theorem which completely characterizes strongly invertible rational links.
}
\end{remark}

\medskip
\begin{theorem}{\em \cite{Diao_Pham2020}}\label{strong_inv}
Let $\mathcal{L}$ be a rational link with two components and is represented by a 4-plat corresponding to $p/q$ (with $q$ even) in the PS form, then $\mathcal{L}$ is strongly invertible if and only if $p/q=[a_1,a_2,...,a_k,1+2\beta,a_k,...,a_2,a_1]$ for some integers $a_1>0$, ..., $a_k>0$, $\beta\ge 0$.
\end{theorem}

\medskip
An immediate consequence of Theorem \ref{strong_inv} is that a two component rational link with an even crossing number is never strongly invertible, and a two component rational link with a 4-plat in the PS form corresponding to a symmetric vector is always strongly invertible (of course it must have an odd crossing number in this case). That is, $|\Lambda^\ast_n|=0$ if $n$ is even and $|\Lambda^\ast_n|
=TL^s_n$ if  $n$  is odd. This leads to the enumeration formula for $|\Lambda_n|$ given in the following theorem .

\medskip
\begin{theorem}\label{general_4plat_count}
The total number of oriented rational links with a given crossing number $n\ge 2$ is given by the following formula:
\begin{eqnarray}\label{T2_eq}
|\Lambda_n|
&=&\left\{
\begin{array}{ll}
\frac{1}{3}(2^{n-1}+1)+2^{\frac{n}{2}-1} & {\rm if\ }n\ {\rm is\ even},\\
\frac{1}{3}(2^{n-1}+2^{\frac{n-1}{2}}-2) & {\rm if\ }n\ {\rm is\ odd\ and\ }n\equiv 1\ {\rm mod(4)},\\
\frac{1}{3}(2^{n-1}+2^{\frac{n-1}{2}}) & {\rm if\ }n\ {\rm is\ odd\ and\ }n\equiv 3\ {\rm mod(4)}.
\end{array}
\right.
\end{eqnarray}
\end{theorem}
Notice that the right side of (\ref{T2_eq}) can be combined into the following single expression
\begin{equation}\label{T2_eq_2}
\frac{1}{3}\left(2^{n-1}+\frac{5+(-1)^n}{2}2^{\lfloor \frac{n}{2}\rfloor-1}+\frac{-1+(-1)^n+
2(-1)^{\lfloor \frac{n+1}{2}\rfloor n}}{2}\right).
\end{equation}

\begin{proof} If $n$ is even, then there are no strongly invertible 4-plats (in  the PS form)  with $n$ crossings. For $n\ge  4$, $TK_n=\frac{1}{3}(2^{n-2}-1)$ and $TL_n=\frac{1}{3}(2^{n-3}+1)+2^{\frac{n-4}{2}}$ by \cite[Theorem 1]{Ernst1987}, thus 
\begin{eqnarray*}
|\Lambda_n|
&=&
TK_n+2TL_n\\
&=&
\frac{1}{3}(2^{n-2}-1)+\frac{1}{3}(2^{n-2}+2)+2^{\frac{n-2}{2}}\\
&=&
\frac{1}{3}(2^{n-1}+1)+2^{\frac{n}{2}-1}.
\end{eqnarray*}
Since there are two oriented rational links with crossing number 2, and the above formula also yields $|\Lambda_2|=2$, so the formula holds in general for any even  integer $n\ge 2$.
This proves the first part of (\ref{T2_eq}). 

\medskip
If $n=4k+1$ for some $k\ge 1$, then by \cite[Theorem 1]{Ernst1987}, we have 
$TK_n=\frac{1}{3}(2^{n-2}+2^{\frac{n-1}{2}})$ and $TL_n=\frac{1}{3}(2^{n-3}+2^{\frac{n-3}{2}})$. Furthermore, by an argument similar to the one used in the proof of \cite[Lemma 2]{Ernst1987}, we also obtain $TL^s_n=\frac{1}{3}(2^{\frac{n-1}{2}}+2)$. It follows that
\begin{eqnarray*}
|\Lambda_n|
&=&
TK_n+2TL_n-TL^s_n\\
&=&
\frac{1}{3}(2^{n-2}+2^{\frac{n-1}{2}})+\frac{1}{3}(2^{n-2}+2^{\frac{n-1}{2}})-\frac{1}{3}(2^{\frac{n-1}{2}}+2)\\
&=&
\frac{1}{3}(2^{n-1}+2^{\frac{n-1}{2}}-2).
\end{eqnarray*}

\medskip
Similarly, if $n=4k+3$ for some $k\ge 1$, we have 
$TK_n=\frac{1}{3}(2^{n-2}+2^{\frac{n-1}{2}}+2)$, and $TL_n=\frac{1}{3}(2^{n-3}+2^{\frac{n-3}{2}}-2)$  and
$TL^s_n=\frac{1}{3}(2^{\frac{n-1}{2}}-2)$. It follows that
\begin{eqnarray*}
|\Lambda_n|
&=&
TK_n+2TL_n-TL^s_n\\
&=&
\frac{1}{3}(2^{n-2}+2^{\frac{n-1}{2}}+2)+\frac{1}{3}(2^{n-2}+2^{\frac{n-1}{2}}-4)-\frac{1}{3}(2^{\frac{n-1}{2}}-2)\\
&=&
\frac{1}{3}(2^{n-1}+2^{\frac{n-1}{2}}).
\end{eqnarray*}
This proves Theorem \ref{general_4plat_count} for the case when $n$ is odd and $n\ge 5$. Since there are two oriented rational links with crossing number 3, the formula given in Theorem \ref{general_4plat_count} also works for this case.
This completes the proof of Theorem \ref{general_4plat_count}. \qed
\end{proof}

\section{The Enumeration of Rational Links with a Given Deficiency Number}\label{enu_def}

Let $D$ be a reduced alternating link diagram of an alternating link $\mathcal{L}$. Let $s(D)$ be the number of Seifert circles in $D$ and let $g(D)$ be the genus of the Seifert surface constructed from the Seifert circle decomposition of $D$.  It is known that $g(D)=g(\mathcal{L})$ \cite{Murasugi1958_1,Murasugi1958_2} and satisfies the equation 
$ 2g(D)=c(D)-s(D)-\mu(\mathcal{L})+2$ hence $0=c(D)-2g(\mathcal{L})-s(D)-\mu(\mathcal{L})+2$
 where $c(D)$ is the number of crossings in $D$, which also equals $Cr(\mathcal{L})$, the minimum crossing number of $\mathcal{L}$.  It follows that 
\begin{eqnarray*}
d(\mathcal{L})
&=&
(Cr(\mathcal{L})-2g(\mathcal{L})-\b(\mathcal{L})-\mu(\mathcal{L})+2)\\
&-&(c(D)-2g(\mathcal{L})-s(D)-\mu(\mathcal{L})+2)\\
&=&s(D)-\b(\mathcal{L}).
\end{eqnarray*}
Thus $\mathcal{L}$ is of deficiency zero if and only if $\b(\mathcal{L})=s(D)$ for any reduced alternating link diagram $D$ of $\mathcal{L}$. Consider a small Seifert circle in the R-decomposition of a 4-plat $\mathcal{L}$ in the PS form. It shares a single crossing with either another small Seifert  circle or a medium Seifert circle on either of its two sides, as shown in the right side of Figure \ref{Reduction}. We call the operation that removes this small Seifert circle and then combine its two neighboring 
Seifert circles into a single Seifert circle a {\em reduction operation}. By a recent result in \cite{DL}, we have $\b(\mathcal{L})=s(\mathcal{L})-r(\mathcal{L})$, where $s(\mathcal{L})$ is the number of Seifert circles in $\mathcal{L}$ and $r(\mathcal{L})$ is the number of reduction operations one can perform on $\mathcal{L}$. By our discussion in the introduction section, we see that $r(\mathcal{L})$ equals the deficiency number $d(\mathcal{L})$ of $\mathcal{L}$. Furthermore, $d(\mathcal{L})=d(\mathcal{L}^\ast)$ if $\mathcal{L}^\ast$ is the mirror image of $\mathcal{L}$. Thus by similar arguments used in Remark \ref{R_counting_2} to establish (\ref{counting2}), we have
\begin{equation}
|\Lambda_n(d)|=|R^I_n(d)|+|R^{III}_n(d)|+|RS^{III}_n(d)|, \label{counting_d}
\end{equation}
where $\Lambda_n(d)$ is the set of all rational links with crossing number $n$ and deficiency number $d$, $R^{I}_n(d)$, $R^{III}_n(d)$ and $RS^{III}_n(d)$ denote the sets of R-decompositions of Type I, Type III and symmetric Type III 4-plats with deficiency $d$ respectively.

\medskip
\begin{figure}[!ht]
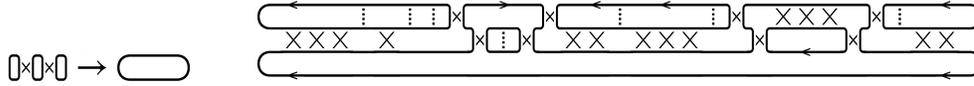

\begin{center}
\includegraphics[scale=.11]{Reduction}\qquad \includegraphics[scale=.11]{Reduction2}
\end{center}
\caption{Left: The reduction operation on a 4-plat in PS form combines three Seifert circles (with the middle one being a small Seifert circle) into one and reduces the number of crossings by 2. Notice that the two Seifert circles on the two sides may be either medium or small Seifert circles; Right: The resulting R-decomposition of the 4-plat in Figure \ref{multiple_blocks} after all seven possible reduction operations are performed, the dashed lines indicate where the reductions were carried out.
\label{Reduction}}
\end{figure}

In the following we explain our approach in the determination of $R^{I}_n(d)$, $R^{III}_n(d)$ and $RS^{III}_n(d)$. 
Start from a Type I  (Type III) R-decomposition with deficiency $d$. If we perform $d$ reduction operations we will end up with a Type I (Type III) R-decomposition with zero deficiency such as  the two R-decompositions  on the right sides  of Figures \ref{Reduction} and \ref{Template}. Notice that in an R-decomposition with deficiency $0$, there are no small Seifert circles and each medium Seifert circle shares at least two crossings with the large Seifert circle. We will call the first crossing and the last crossing that each medium Seifert circle shares with the large Seifert circle the {\em essential crossings}. Notice that the inverse operation of a reduction operation is to first split a medium or small Seifert circle into two Seifert circles at a location that will not affect the essential crossings, and then insert a small Seifert circle between them (and  adding  a crossing between the  newly created small Seifert circle and  each  of  its  two  neighboring Seifert  circles). We call this an {\em insertion operation}. 
The crossings in an R-decomposition with deficiency $0$ that are not  essential are called {\em free crossings} and the crossings deleted by the reduction operations are called {\em $r$-crossings}. A reduction  operation reduces the number of medium/small Seifert  circles and the number of $r$-crossings each by 2, while an insertion operation increases the number of  medium/small Seifert circles and the number of $r$-crossings each by 2.

\medskip
An R-decomposition with neither free crossings nor $r$-crossings is called an {\em R-template}. An example of an R-template is given in the left of Figure \ref{Template}, which can be obtained from the R-decomposition in the right of Figure \ref{Template} by removing the free crossings. Since an R-decomposition with deficiency $d$ and $f$ free crossings can be reduced to an R-template by the reduction operations and free crossing deletions, it can also be re-constructed from an R-template. For example, the R-decomposition in the right of Figure \ref{Template} can be constructed from the R-template given in the left of Figure \ref{Template} by adding 1, 3, 0 and 2 free crossings in the shaded areas from left to right respectively as shown in Figure \ref{Template}.  The R-decomposition with zero deficiency shown in the  right of Figure \ref{Reduction} is constructed from the Type III template with 5 medium Seifert circles by adding 12 free crossings in the positions shown, and the R-decomposition in Figure \ref{multiple_blocks}  can be re-constructed from it by adding the small Seifert circles and crossings in  the locations indicated by dashed  lines in the figure. In general,  if the R-template contains $k$ medium Seifert circles, then it contains $2k$ essential crossings. Since each reduction operation reduces 2 crossings, if an R-decomposition with deficiency $d$ and $f$ free crossings is reduced to an R-template with $k$ medium Seifert circles, then we have $f+2d+2k=n$. 

\medskip
\begin{figure}[!ht]
\begin{center}
\includegraphics[width=2.4in,height=0.5in]{Figure11}\quad \includegraphics[width=2.4in,height=0.5in]{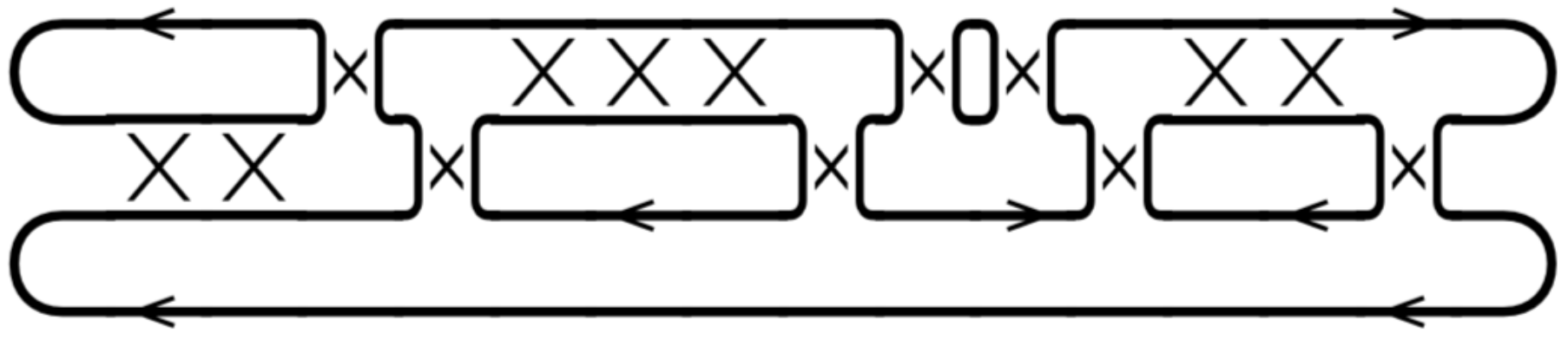}
\end{center}
\caption{Left: A Type I R-template with four medium Seifert circles (and 8 essential crossings) corresponding to the signed vector $(1,1,-2,2,-2)$; Right: A Type I R-decomposition with four medium Seifert circles and 14 crossings can be constructed from the Type I R-template in  the left by adding $6$ crossings in the shaded areas.
\label{Template}}
\end{figure}

\subsection{The counting of Type I R-decompositions.}

Consider an R-decomposition in $R^{I}_n(d)$ whose corresponding R-template has $k=2j$ ($j\ge 1$) medium Seifert circles. The number of free crossings in it is given by  $f=n-2d-2k=n-2d-4j$ (it is necessary that $n\ge 4$ in this case). The free crossings can be distributed to the $2j$ horizontal spaces between the medium Seifert circles and the large Seifert circle (marked by the gray boxes in Figure \ref{Template}), and there are $C(f+2j-1,2j-1)$ ways to do so. Once the free crossings have been determined, there are $2j+f$ slots to insert the small Seifert circles as indicated by the vertical dashed lines in Figure  \ref{Reduction}, and there are a total of $C(d+f+2j-1,f+2j-1)$ ways to  perform the $d$ insertions. It follows that the total number of Type I R-decompositions with $n$ crossings, deficiency $d$ that can be constructed from an R-template with  $k=2j\ge 4$ medium Seifert circles  is $C(f+2j-1,2j-1)C(d+f+2j-1,f+2j-1)=C(f+2j-1,2j-1)C(d+f+2j-1,d)$ where  $f=n-2d-4j$. It follows that 

\medskip
\begin{equation}\label{TypeI_count}
|R^I_n(d)|=\sum_{j=1}^{\lfloor \frac{n-2d}{4}\rfloor}
{{n-2d-2j-1}\choose{2j-1}}{{n-d-2j-1}\choose{d}}.
\end{equation}

For the sake of convenience we define any summation term in the above as 0 if $\lfloor \frac{n-2d}{4}\rfloor<1$.

\medskip
\subsection{The counting of Type III R-decompositions.}

\medskip
A Type III R-decomposition has an R-template with an odd number (at least 1) of medium Seifert circles (hence at least 2 crossings). Similar to the discussion of  the Type I case above, we have
\begin{equation}\label{TypeIII_count}
|R^{III}_n(d)|=\sum_{j=0}^{\lfloor \frac{n-2d-2}{4}\rfloor}
{{n-2d-2j-2}\choose{2j}}{{n-d-2j-2}\choose{d}}.
\end{equation} 
The details are left to the reader.

\medskip
\subsection{The counting of symmetric Type III R-decompositions.}

\medskip
Consider an R-decomposition in $RS^{III}_n(d)$ whose corresponding R-template has $1+2j$ ($j\ge 0$) medium Seifert circles and deficiency $d$. There are $f=n-4j-2d-2$ free crossings to be distributed to  the $1+2j$ available  slots in a  symmetric way. 

\medskip
Case 1. $n$ is even. In this case $f$ is even and there are $C(n/2-j-d-1,j)$ distinct ways to distribute  the free crossings in the template symmetrically and each of them results  in an  R-decomposition in $RS^{III}_{n-2d}(0)$ with a total of $1+2j+f$ slots for the $d$ insertion operations, and the  insertions  need to be performed in  a symmetric manner as well.

\medskip
Case 1(a) $d$ is even. In  this case for each R-decomposition from the above, there are $C((n-d)/2-j-1,d/2)$ distinct ways to distribute  the  $d$ insertions symmetrically. 

\medskip
Case 1(b)  $d$ is odd (so  $d\ge  1$). Here for each R-decomposition in Case 1, there are $C((n-d-1)/2-j-1,(d-1)/2)$ distinct ways to distribute the  $d$ insertions symmetrically.

\medskip
Case 2. $n$ is odd. In this case $f$ is odd and there are $C((n-1)/2-j-d-1,j)$ distinct ways to distribute  the free crossings in the template symmetrically and a free crossing has to be placed in the middle of the template. Each of which leads to an  R-decomposition in $RS^{III}_{n-2d}(0)$ which has a total of $1+2j+f$ slots for the $d$ insertion operations to be performed in  a symmetric manner.

\medskip
Case 2(a) $d$ is even. In  this case for each R-decomposition from the above, there are $C((n-d-1)/2-j-1,d/2)$ distinct ways to distribute the  $d$ insertions symmetrically. 

\medskip
Case 2(b)  $d$ is odd (so  $d\ge  1$). In  this case for it is not possible  to distribute the  $d$ insertions symmetrically since to do so an insertion  has to be in the middle  of  the  template, yet the  middle  has been occupied by a free crossing already.

\medskip
Summarizing the above cases,  we have 
\begin{eqnarray}
&&|RS^{III}_n(d)|\nonumber\\
&=&\left\{
\begin{array}{ll}
\sum_{j=0}^{\lfloor \frac{n-2d-2}{4}\rfloor}{{\frac{n}{2}-d-j-1}\choose{j}}
{{\frac{n-d}{2}-j-1}\choose{\frac{d}{2}}},&\ n\ {\rm even},\  d\ {\rm even}\\
\sum_{j=0}^{\lfloor \frac{n-2d-2}{4}\rfloor}{{\frac{n}{2}-d-j-1}\choose{j}}
{{\frac{n-d-1}{2}-j-1}\choose{\frac{d-1}{2}}},&\ n\ {\rm even},\  d\ {\rm odd}\\
\sum_{j=0}^{\lfloor \frac{n-2d-2}{4}\rfloor}{{\frac{n-1}{2}-d-j-1}\choose{j}}
{{\frac{n-d-1}{2}-j-1}\choose{\frac{d}{2}}},&\ n\ {\rm odd},\  d\ {\rm even}\\
0,&\ n\ {\rm odd},\  d\ {\rm odd}
\end{array}
\right.\nonumber\\
&=&
\frac{1+(-1)^{nd}}{2}\sum_{j=0}^{\lfloor \frac{n-2d-2}{4}\rfloor}{{\lfloor\frac{n}{2}\rfloor-d-j-1}\choose{j}}
{{\lfloor\frac{n-d}{2}\rfloor-j-1}\choose{\lfloor\frac{d}{2}\rfloor}}.
\label{RSIII_count}
\end{eqnarray}

\medskip
Table \ref{t1} contains the computation results of $|R^I_n(d)|$, $|R^{III}_n(d)|$, $|RS^{III}_n(d)|$  and $|\M_n|$ for $2\le n\le 13$. 

\begin{table}[h]
\begin{center}
\begin{tabular}{|c|c|c|c|c|c|c|c|}\hline
\backslashbox{$n$}{$d$}&0&1&2&3&4&5&$|\M_n|$\\\hline\hline
2&0,1,1&&&&&&2\\\hline
3&0,1,1&&&&&&2\\\hline
4&1,1,1&0,1,1&&&&&5\\\hline
5&2,1,1&0,2,0&&&&&6\\\hline
6&3,2,2&2,3,1&0,1,1&&&&15\\\hline
7&4,4,2&6,4,0&0,3,1&&&&24\\\hline
8&6,7,3&12,8,2&3,6,2&0,1,1&&&51\\\hline
9&10,11,3&20,18,0&12,10,2&0,4,0&&&90\\\hline
10&17,17,5&34,37,3&30,21,5&4,10,2&0,1,1&&187\\\hline
11&28,27,5&62,68,0&60,51,5&20,20,0&0,5,1&&352\\\hline
12&45,44,8&116,119,5&115,118,10&60,45,5&5,15,3&0,1,1&715\\\hline
13&72,72,8&212,208,0&228,246,10&140,116,0&30,35,3&0,6,0&1386\\\hline
\end{tabular}
\end{center}
\caption{The three numbers in  the $(n,d)$ position are $|R^I_n(d)|$, $|R^{III}_n(d)|$ and $|RS^{III}_n(d)|$ respectively in that order.} \label{t1}
\end{table}

\section{Further Discussion: Simplifications and Relation to Known Number Sequences}

\medskip
In this section we  take a closer look at the formulas we obtained and show that they can be related to some well known integer sequences. We have the following theorem.

\begin{theorem}\label{equiv_thm}
Let $F^{(d)}_{q}$ be the convolved Fibonacci sequence defined (recursively) by
$F^{(d)}_{q}=\sum_{k=0}^q F_k\cdot F_{q-k}^{(d-1)}$ for $d\ge 1$, and $F^{(0)}_{q}=F_q$ where $\{F_0,F_1,F_2,F_3,F_4,F_5,...\}=\{0,1,1,2,3,5,...\}$ is the Fibonacci sequence, then we have
\begin{equation}\label{Fibo_eq_1}
|R^I_n(d)|+|R^{III}_n(d)|=F^{(d)}_{n-d-1}
\end{equation}
and
\begin{eqnarray}\label{Fibo_eq_2}
|RS^{III}_n(d)|
&=&
\frac{1+(-1)^{nd}}{2}F^{(\lfloor \frac{d}{2}\rfloor)}_{\lfloor \frac{n}{2}\rfloor-\lfloor \frac{d+1}{2}\rfloor}.
\end{eqnarray}
\end{theorem}

\begin{proof} We will prove (\ref{Fibo_eq_1}) first. 
For the sake of convenience we will use  the notation $H^{(d)}_{n}$ to denote $|R^I_n(d)|+|R^{III}_n(d)|$. By making the substitution $k=2j-1$ in (\ref{TypeI_count}) and $k=2j$ in (\ref{TypeIII_count}), 
we obtain
\begin{equation}\label{TypeIandIII_count}
H^{(d)}_{n}
=\sum_{k=0}^{\lfloor \frac{n-2d-2}{2}\rfloor}
{{n-2d-2-k}\choose{k}}{{n-d-2-k}\choose{d}}.
\end{equation}
Since there are no nontrivial rational links with crossing number $n\ge 1$ nor with deficiency $d>(n-2)/2$, we will define $H^{(d)}_{0}=H^{(d)}_{1}=0$ for any $d$ and $H^{(d)}_{n}=0$ for any $d$ such that $d>(n-2)/2$ or $d<0$. 
The equation $H^{(d)}_{n}= F^{(d)}_{n-d-1}$ holds for small values of $n$ and $d$ by Table \ref{t1}. It is known that $F_n^{(d)}$ can be determined by the recursive relation
$$
F_{n}^{(d)}=F_{n-1}^{(d)}+F_{n-2}^{(d)}+F_{n-1}^{(d-1)},
$$
this is equivalent to the following recurrence equation for $H_n^{(d)}$  ($n\ge 2$ and $d\ge -1$):
\begin{equation}\label{conv_rec}
H_{n+1}^{(d+1)}=H_{n}^{(d+1)}+H_{n-1}^{(d+1)}+H_{n-1}^{(d)}.
\end{equation}
We will verify (\ref{conv_rec}) by 
considering the 2 different cases $n=2m$ and $n=2m+1$ ($m\ge 1$). In the following we shall demonstrate the  case  of $n=2m$. The other case can  be similarly verified and is left to the  reader  as an exercise. Write $m-d-2=q$ for short, we have
\begin{eqnarray*}
&&H_{2m}^{(d+1)}+H_{2m-1}^{(d+1)}\\
&=&
\sum_{k=0}^{q}
{{2q-k}\choose{k}}{{2q+d+1-k}\choose{d+1}}+
\sum_{k=0}^{q-1}
{{2q-1-k}\choose{k}}{{2q+d-k}\choose{d+1}}\\
&=&
{{2q+d+1}\choose{d+1}}+
\sum_{k=0}^{q-1}
{{2q-k-1}\choose{k+1}}{{2q+d-k}\choose{d+1}}\\
&+&
\sum_{k=0}^{q-1}
{{2q-1-k}\choose{k}}{{2q+d-k}\choose{d+1}}\\
&=&
{{2q+d+1}\choose{d+1}}+
\sum_{k=0}^{q-1}
{{2q-k}\choose{k+1}}{{2q+d-k}\choose{d+1}}\\
&=&
\sum_{k=0}^{q}
{{2q+1-k}\choose{k}}{{2q+d+1-k}\choose{d+1}}.
\end{eqnarray*}
It follows that
\begin{eqnarray*}
&&H_{2m}^{(d+1)}+H_{2m-1}^{(d+1)}+H_{2m-1}^{(d)}\\
&=&
\sum_{k=0}^{q}
{{2q+1-k}\choose{k}}{{2q+d+1-k}\choose{d+1}}\\
&+&
\sum_{k=0}^{q}
{{2q+1-k}\choose{k}}{{2q+d+1-k}\choose{d}}\\
&=&
\sum_{k=0}^{q}
{{2q+1-k}\choose{k}}{{2q+d+2-k}\choose{d+1}}\\
&=&
H_{2m+1}^{(d+1)}.
\end{eqnarray*}
Notice that the above holds for the special case $d=-1$ as well since in which case $H_{2m-1}^{(-1)}=0$ and we only need to consider  the first two  summations in the above proof (and the equation still holds in which case). This proves (\ref{Fibo_eq_1}).

\medskip
We now prove (\ref{Fibo_eq_2}). For $d=2c$, $n=2m\ge 2+2c$, we have 
\begin{eqnarray*}
&&|RS^{III}_n(d)|=|RS^{III}_{2m}(2c)|\\
&=&
\sum_{j=0}^{\lfloor \frac{2m-4c-2}{4}\rfloor}{{\lfloor\frac{2m}{2}\rfloor-2c-j-1}\choose{j}}
{{\lfloor\frac{2m-2c}{2}\rfloor-j-1}\choose{\lfloor\frac{2c}{2}\rfloor}}\\
&=&
\sum_{j=0}^{\lfloor \frac{m-2c-1}{2}\rfloor}{{m-2c-j-1}\choose{j}}{
{m-c-j-1}\choose{c}}\\
&=&H_{m+1}^{(c)}=F^{(c)}_{m-c}=F^{(\lfloor \frac{d}{2}\rfloor)}_{\lfloor \frac{n}{2}\rfloor-\lfloor \frac{d+1}{2}\rfloor}
\end{eqnarray*}
by the proof above for (\ref{Fibo_eq_1}).
Similarly one can prove that 
$$
|RS^{III}_{2m+1}(2c)|=F_{m-c}^{(c)}=F_{\lfloor \frac{n}{2}\rfloor-\lfloor \frac{d+1}{2}\rfloor}^{(\lfloor \frac{d}{2}\rfloor)}
$$
and 
$$
|RS^{III}_{2m}(2c+1)|=F_{m-c-1}^{(c)}=F_{\lfloor \frac{n}{2}\rfloor-\lfloor \frac{d+1}{2}\rfloor}^{(\lfloor \frac{d}{2}\rfloor)}.
$$
The details are left to the reader.
\qed
\end{proof}

\medskip
Since for each fixed $n$, the summation of $|R^I_n(d)|$, $|R^{III}_n(d)|$, $|RS^{III}_n(d)|$ over $d$ equals $|\M_n|$,  Theorem \ref{equiv_thm} leads to the following equality.

\begin{corollary}
For any $n\ge 2$ and $d\ge 0$, we have 
\begin{eqnarray}
&&
\frac{1}{3}\left(2^{n-1}+\frac{5+(-1)^n}{2}2^{\lfloor \frac{n}{2}\rfloor-1}+\frac{-1+(-1)^n+
2(-1)^{\lfloor \frac{n+1}{2}\rfloor n}}{2}\right)\nonumber\\
&=&
\sum_{d=0}^{\lfloor \frac{n-2}{2}\rfloor}\left(F_{n-d-1}^{(d)}+
\frac{1+(-1)^{nd}}{2}F^{(\lfloor \frac{d}{2}\rfloor)}_{\lfloor \frac{n}{2}\rfloor-\lfloor \frac{d+1}{2}\rfloor}\right).\label{n_even_odd}
\end{eqnarray}
\end{corollary}

\medskip
We note that  the left side of (\ref{n_even_odd}) is the sequence A007581 (with its first term truncated) in the Online Encyclopedia of Integer Sequences~\cite{OEIS} when $n$ is even, and is A192466 when $n$ is odd.

\medskip
We end our paper with the following remark.

\medskip
\begin{remark}{\em
In \cite{DiaoErnst2005}, it was established that $|\M_n(0)|$ grows exponentially at a rate at least $e^{\frac{\ln 3}{7} n}=(\sqrt[7]{3})^n\approx 1.17^n$. In the special case of $d=0$, Theorem \ref{equiv_thm} yields $|\M_n(0)|=F_{n-1}+F_{\lfloor \frac{n}{2}\rfloor}$. Since $F_n=\frac{\phi^n-(-\phi)^{-n}}{\sqrt{5}}$ where $\phi=\frac{1+\sqrt{5}}{2}\approx 1.618$ is the golden ratio, this improves this growth rate of $|\M_n(0)|$ to approximately $\frac{1.618^{n-1}+1.272^n}{\sqrt{5}}$.
}
\end{remark}

\section*{Acknowledgement} 

The authors thank Prof. Gabor Hetyei for his helpful comments and suggestions.


\medskip
\section*{References}
\begin{thebibliography}{99}
\bibitem{A} C.~Adams, {\em The Knot Book}, American Mathematical Soc., 1994.

\bibitem{BZ} G.~ Burde, H.~ Zieschang and M.~ Heusener
{\em Knots}, De Gruyter Studies in Mathematics \textbf{5},  2013.

\bibitem{Co}
J.~H.~ Conway,  An enumeration of knots and links, and some of their algebraic
properties, 1970 Computational Problems in Abstract Algebra
(Proc. Conf., Oxford, 1967) pp. 329--358 Pergamon, Oxford.

\bibitem{Cromwell} P.~ Cromwell, Knots and links, Cambridge University Press, 2004.

\bibitem{Diao2004} Y.~ Diao, The Additivity of Crossing Numbers, {\em Journal of Knot Theory and its Ramifications} \textbf{13} 7 (2004), 857--866.

\bibitem{DiaoErnst2005} Y.~Diao and C.~Ernst, The Growth Rate of Some Deficiency Zero Knot Classes, {\em International Journal of Pure and Applied Mathematics} \textbf{23} 4 (2005), 437--450.

\bibitem{DL}
  Y.~ Diao, C.~ Ernst, G.~ Hetyei and P.~ Liu,
A diagrammatic approach for determining the braid index of alternating
links, preprint 2018. 


\bibitem{Diao_Pham2020} Y.~Diao G.~ Hetyei and V.~Pham, A writhe-like invariant for alternating link diagrams, preprint.

\bibitem{Doll} H.~Doll and J.~Hoste, A tabulation of oriented links. {\em Math. Computat.}, \textbf{57}(196) (1991), 747--761.

\bibitem{Ernst1987} 
C.~Ernst and D.~Sumners, The growth of the number of prime knots. {\em Mathematical Proceedings of the Cambridge Philosophical Society}, \textbf{102}(2) (1987), 303--315. 

\bibitem{Hoste1998} J.~Hoste, M.~B.~Thistlethwaite and J.~Weeks, The first 1,701,936 knots, {\em Math. Intelligencer} \textbf{20}(4) (1998), 33--48.
  
\bibitem{Ka-La}
L.H.~ Kauffman, and S.~ Lambropoulou, 
{\em On the classification of rational knots,}
Enseign. Math \textbf{2} (2002), 357--410.   



\bibitem{Murasugi1958_1} K.~Murasugi, On the genus of the alternating knot I, {\em J. Math. Soc. Japan}, \textbf{10}(1) (1958), 94--105.

\bibitem{Murasugi1958_2} K.~Murasugi, On the genus of the alternating knot II, {\em J. Math. Soc. Japan}, \textbf{10}(3) (1958), 235--248.

\bibitem{OEIS}
OEIS Foundation Inc.\ (2011), {\em The On-Line Encyclopedia of Integer Sequences,}
published electronically at \url{http://oeis.org}.

 \bibitem{Schubert}
 K.~ Schubert, Knoten mit zwei Br{\"u}cken, Mathematische Zeitschrift, \textbf{65} (1956), 133--170.
 
\end{thebibliography}
\end{document}